\documentclass[journal,12pt,draftclsnofoot,onecolumn]{IEEEtran}

\ifCLASSINFOpdf
\else
\fi

\usepackage{amsmath}
\usepackage{cite}
\usepackage{amsfonts}
\usepackage{amssymb}
\usepackage{caption}
\usepackage[font=small]{caption}

\usepackage[square, comma, sort&compress,numbers]{natbib} 

\usepackage{array}
\usepackage[margin=1.0in]{geometry}
\usepackage{amsthm}
\renewcommand*{\stackrel}{%
\mathrel\bgroup\stack@relbin
}

\newtheorem{theorem}{Theorem}[section]
\newtheorem{lemma}[theorem]{Lemma}

\ifCLASSINFOpdf
   \usepackage[pdftex]{graphicx}
  \graphicspath{{../pdf/}{../jpeg/}{../eps/}}
  \DeclareGraphicsExtensions{.eps,.pdf,.jpeg,.png}
\else
   \usepackage[dvips]{graphicx}
   \graphicspath{{../eps/}}
   \DeclareGraphicsExtensions{.eps}
\fi
\usepackage{epsfig}
\usepackage{fmtcount}
\usepackage[linesnumbered,ruled,vlined]{algorithm2e}

\usepackage{float}

\usepackage{subfigure}
\usepackage{array}
\usepackage{url}

\usepackage{setspace}

\begin{document}\fontsize{12}{12}\rm
%

\title{{\fontsize{17}{15}\selectfont Asynchronous Systems and Binary Diagonal Random Matrices: A Proof and Convergence Rate}}

%
%
%

\author{Syed~Amaar~Ahmad~\IEEEmembership{}\\
Wireless@Virginia Tech,\\ 
saahmad@vt.edu
\thanks{The author is with the Department of Electrical and Computer Engineering
at Virginia Tech, Blacksburg, VA, USA.}
\thanks{}}


%
%

\markboth{ ~}%
{Shell \MakeLowercase{\textit{et al.}}: Bare Demo of IEEEtran.cls for Journals}
%





\maketitle

\begin{abstract}
In a synchronized network of $n$ nodes, each node will update its parameter based on the system state in a given iteration. It is well-known that the updates can converge to a fixed point if the maximum absolute eigenvalue (spectral radius) of the $n \times n$ iterative matrix ${\bf{F}}$ is less than one (i.e. $\rho({\bf{F}})<1$).  However, if only a subset of the nodes update their parameter in an iteration (due to delays or stale feedback) then this effectively renders the spectral radius of the iterative matrix as one. We consider matrices of unit spectral radii generated from ${\bf{F}}$ due to random delays in the updates. We show that if each node updates at least once in every $T$ iterations, then the product of the random matrices (joint spectral radius) corresponding to these iterations is less than one. We then use this property to prove convergence of asynchronous iterative systems. Finally, we show that the convergence rate of such a system is $\rho({\bf{F}})\frac{(1-(1-\gamma)^T)^n}{T}$, where assuming ergodicity, $\gamma$ is the lowest bound on the probability that a node will update in any given iteration.
\end{abstract}

%
\IEEEpeerreviewmaketitle

\section{Introduction}
Iterative systems naturally arise in wireless networks, parallel processors or in game-theoretic applications when multiple
independent agents or nodes update their parameters based on their observations of the system. In a \emph{synchronized} system, all nodes perform these updates in every iteration. Conversely, we define an \emph{asynchronous system}, where either (i) only a subset of  nodes will update in an iteration, or (ii) all nodes may update but the  adaptation of some nodes is based on stale system information based on a previous iteration. Either situation may arise due to random feedback delays on receiving the system state information. 

It is well-known that if the spectral radius of the underlying iterative matrix (which determines the interference or interaction between each pair of nodes) is less than one then the synchronized system converges to a \emph{fixed point}. In the following discussion, we provide a proof for convergence of an asynchronous iterative system. This proof is based on showing that the product of the effective matrices within any $T$ iterations has a spectral radius is less than one even though the spectral radii of individual matrices may be one. Finally, we also show convergence of the system given estimation error as long as (i) the error is independent of the nodal updates, and (ii) the error projects the iterative matrix into a new matrix which still satisfies the inequality on its spectral radius.

\section{System Description}
Suppose a linear system comprising $n$ nodes which are all independently updating some parameter. Each node updates its parameter $P_i$ based on its received information regarding the state of the system from its vantage point $E_i$. The nodes update their parameters in iterations $k \in \{1,2,\cdots, \infty\}$. In iteration $k$, we define:
\begin{equation}
E_i(k) = \sum_{j\neq i} g_{i,j}P_j(k),
\end{equation}
where $g_{i,j}$ captures the cross-talk or interference effect between node $i$ and node $j$ where in general $g_{j,i}\neq g_{i,j}$. The update of each node $i$ is such that:
\begin{equation}
\begin{split}	
P_i(k+1) &= D_i + E_i(k)\\
&= D_i + \sum_{j\neq i} g_{i,j}P_j(k).
\end{split}
\end{equation}
On the other hand $D_i$ represents some fixed parameter particular to node $i$. We let $\bf{D}$ be $n \times 1$ vector where element $i$ such that ${\bf{D}}(i)= D_i$ and define an $n \times n$ matrix ${\bf{F}}$ such that its element in row $i$ and column $j$ given by
\begin{equation}
\begin{split}
{\bf{F}}(i,j) = g_{i,j}. \label{matrix}
\end{split}
\end{equation}
Note that depending on the system, the $i^{th}$ diagonal entry in the above matrix may be typically be zero (i.e. ${\bf{F}}(i,i)=0$) when there is no auto-feedback for the node $i$. Each node $i$, which is assumed to know its current state in terms of $E_i(k)$, updates its paramter $P_i$ based on this observation. Note that $P_i(k+1)$ denotes its updated value for iteration $k+1$. In vector notation, we can describe the system state alternatively as
\begin{equation}\label{cutoff_thre}
{\bf{E}}(k)= {\bf{FP}}(k)
\end{equation}
and the adaptations or updates by the nodes equivalently as
\begin{equation}
\begin{split}
{\bf{P}}(k+1) &= {\bf{D}}+ {\bf{E}}(k)\\
&={\bf{D}} +{\bf{FP}}(k). \label{matrixadapt}
\end{split}
\end{equation}

\section{Synchronous Convergence}
Let ${\bf{I}}$ denote the identity matrix and $\rho({\bf{Q}})$ be the spectral radius of a $n \times n$ matrix ${\bf{Q}}$.

\begin{lemma}
If $\rho({\bf{F}})<1$, then a unique feasible fixed point for the nodal updates is $\left[{\bf{I+F}}\right]^{-1}{\bf{D}}$.
\end{lemma}
\begin{proof}
It is well-known that the evolution of the updates in \eqref{matrixadapt} is such that:
\begin{equation}\label{powerseries}
\begin{split}
{\bf{P}}(k+1)&={\bf{D}}-{\bf{F}}\left({\bf{D}} + {\bf{F}}\left({\bf{D}}+{\bf{F}}\left(\cdots {\bf{P}}(0) \right) \right)\right) \\
&={[}{\bf{I}} + {\bf{F}}+{\bf{F}}^2 + {\bf{F}}^3+\cdots\\
&+(-1)^{k-1}{\bf{F}}^{(k-1)}{]}{\bf{D}}+(-1)^k{\bf{F}}^k{\bf{P}}(0)\\
\lim\limits_{k \rightarrow \infty} {\bf{P}}(k+1)&=\left[{\bf{I+F}}\right]^{-1}{\bf{D}}
\end{split}
\end{equation}
where ${\bf{P}}(0)$ represents the initial values chosen by the $n$ nodes. Since the maximum absolute eigenvalue of $\bf{F}$ is less than one then, by definition, the term $\lim_{k \rightarrow \infty} {\bf{F}}^k{\bf{P}}(0)$ will disappear to an all zeros vector \cite[pg. 618, 7.10.10]{meyer04}. 
Finally, as the converged transmit power vector does not depend on the initial transmit powers, the fixed point above is also unique.
\end{proof}

\begin{lemma}
If $\max\limits_{{\bf{F}}}|\lambda_{{\bf{F}}}|<1$ then $\max\limits_{-{\bf{F}}}|\lambda_{-{\bf{F}}}|<1$.
\end{lemma}
\begin{proof}
By definition ${\bf{Fx}}= \lambda{\bf{x}}$ where for a given eigenvalue $\lambda$, $\bf{x}$ is the associated eigenvector. Multiplying the equation by $-1$ would yield $(-{\bf{F}}){\bf{x}}= (-\lambda){\bf{x}}$. Since this does not change the absolute value of $-\lambda$
hence $|\lambda_{(-{\bf{F}})}|=|\lambda_{{\bf{F}}}|$.
\end{proof}
\qedhere
We know from \cite[p 184]{lancaster} that when all the eigenvalues of a square matrix $\bf{Q}$ satisfy the condition $|\lambda_{\bf{Q}}|<1$ then the matrix series ${\bf{I}}+{\bf{Q}}+{\bf{Q}}^2+\cdots = [{\bf{I}}-{\bf{Q}}]^{-1}$. We can therefore substitute ${\bf{Q}}=-{\bf{F}}$ when the spectral radius of $\bf{F}$ is less than one, to get $\left[{\bf{I}} - {\bf{F}}+{\bf{F}}^2+\cdots\right]=[{\bf{I}} +{\bf{F}} ]^{-1}$. Thus, convergence is not affected whether the eigenvalues are positive or negative as long as their absolute value remain bounded by 1. 

\section{Asynchronous Systems}
Thus far, we have assumed that in every iteration all $n$ nodes update their values. Now suppose that only a subset of the $n$ nodes will update in an iteration. Moreover, the subset of updating nodes may change from iteration to iteration. 

\subsection*{Binary Diagonal Random Matrices}
A matrix may be considered \emph{random} if its entries consist of random numbers from some specified distribution \cite{anderson_random_matrices}. We define a binary diagonal random matrix ${\bf{A}}(k)= diag[a_1(k), \cdots, a_n(k)]$ as an $n \times n$ diagonal matrix where $a_i(k) \in \{1,0\}$. If node $i$ updates in iteration $k$ then $a_i(k)=1$, otherwise $a_i(k)=0$. The probability mass function of the values of $a_i(k), \forall i, k$ could be arbitrary.

Given ${\bf{A}}(k)$, the asynchronous iterative system can be described as:
\begin{eqnarray}\
\begin{split} \label{asynclocaladapt}
{\bf{P}}(k+1) &= [{\bf{I}}-{\bf{A}}(k)]{\bf{P}}(k) + {\bf{A}}(k)[{\bf{D}} + {\bf{F}}{\bf{P}}(k)] \\
& = {\bf{D}}(k) + {\bf{F}}(k){\bf{P}}(k)
\end{split}
\end{eqnarray}
where 
\begin{equation}
{\bf{D}}(k) = {\bf{A}}(k){\bf{D}}
\end{equation}
and the modified random iterative matrix ${\bf{F}}(k)$ as
\begin{equation}
{\bf{F}}(k) = {\bf{A}}(k){\bf{F}}+[{\bf{I}}-{\bf{A}}(k)]. \label{FK}
\end{equation}
Note that ${\bf{F}}(k)$ can be considered a \emph{random matrix} since its entries are based on any any arbitrary delay distribution in the nodes' update. 

In \eqref{asynclocaladapt}, a node updates in iteration $k$ or maintains its value from the previous iteration $P_i(k+1) = P_i(k)$. If all nodes update in iteration $k$ then ${\bf{A}}(k)={\bf{I}}$ and we have ${\bf{P}}(k+1)={\bf{D}}+{\bf{F}}{\bf{P}}(k)$. Conversely, no node updates in the iteration then ${\bf{A}}(k)=0.{\bf{I}}$ and thus ${\bf{P}}(k+1)={\bf{P}}(k)$ (all other cases being intermediate situations). 

\begin{lemma}
If $\rho({\bf{F}})<1$ then $\rho({\bf{F}}(k))=1$ if $a_i(k)=0$ for any node $i$.
\end{lemma}
\begin{proof}
We are given that ${\bf{F}}(k)={\bf{A}}(k){\bf{F}} +[{\bf{I}}-{\bf{A}}(k)]$ and the spectral radius of the matrix is such that $\rho({\bf{F}})<1$. Therefore, if $a_i(k)=0$, then row $i$ in matrix ${\bf{F}}(k)$ will all have zeros elements except for the diagonal element ${\bf{F}}(k)(i,i)=1$. As per Gershgorin circle theorem \cite{meyer04}, by definition the matrix ${\bf{F}}(k)$ now has an eigenvalue of $1$ (i.e. $\rho({\bf{F}}(k))=1$).
\end{proof}
In other words, if any node does not update in iteration $k$ then this renders the spectral radius of matrix ${\bf{F}}(k)$ equal to one.

Next assume that each node updates at least once within any $T$ consecutive iterations. This constraint can be captured as $\sum_{t = k-T}^k{\bf{A}}(t) \geq {\bf{I}}$ where the diagonal indicates the total number of times each node has updated between the current iteration $k$ and the preceding $T$ iterations (i.e. interval $t \in \{k, k-1, \cdots, k-T+1\}$). Such a constraint can be considered as a bound on the   random delays in the updates. 

\begin{theorem}
If $\rho(\bf{F})<1$, and $\sum_{t = k-T}^k{\bf{A}}(t) \geq {\bf{I}}$ then $\rho\left(\prod^k_{t=k-T+1}{\bf{F}}(t)\right)<1$. \label{jointSpecRad}
\end{theorem}
\begin{proof}
If the spectral radius of $\bf{F}$ is less than one (i.e. $|\lambda_{\bf{F}}|<1$) then $\lim\limits_{k \rightarrow \infty} {\bf{F}}^k = 0. {\bf{I}}$ \cite[pg. 618]{meyer04}. 
Recall that ${\bf{F}}(t) = {\bf{A}}(t){\bf{F}}+[{\bf{I}}-{\bf{A}}(t)]$ where $k-T+1\leq t\leq k$. 
Next, note that column $i$ of the matrix $\lim\limits_{k \rightarrow \infty} {\bf{F}}_t^k$ will converge to an all zeros vector if $a_i(t) = {\bf{A}}(t)(i,i)=1$. Conversely, if $a_i(t)={\bf{A}}_t(i,i)=0$, then column $i$ of $\lim\limits_{k \rightarrow \infty} {\bf{F}}_t^k$ will not converge to a zeros vector (see \cite[pg. 630]{meyer04} for a detailed discussion). 

Next consider the constraint that link $i$ adapt within $T$ iterations. If the diagonal entry ${\bf{A}}(t)(i,i) = 0$, then there must be some other ${\bf{A}}(u)(i,i) = 1$ where $u \in \{t,t+1, \cdots t+T\}: u \neq t$. Thus, column $i$ of the corresponding matrix $\lim\limits_{k \rightarrow \infty} {\bf{F}}_u^k$ would be an all zeros vector. Consequently, over $T$ iterations, we deduce that
\begin{equation}
\begin{split}
\lim\limits_{k \rightarrow \infty} {\bf{F}}(t)^k\lim\limits_{k \rightarrow \infty} {\bf{F}}({t-1})^k\cdots \lim\limits_{k \rightarrow \infty} {\bf{F}}({t-T})^k = 0. {\bf{I}}\\
\lim\limits_{k \rightarrow \infty} {\bf{F}}(t)^k{\bf{F}}({t-1})^k\cdots {\bf{F}}({t-T})^k = 0. {\bf{I}}\\
\lim\limits_{k \rightarrow \infty} \left({\bf{F}}(t){\bf{F}}({t-1})\cdots {\bf{F}}({t-T})\right)^k = 0. {\bf{I}}\\
\end{split}
\end{equation}
By definition, the spectral radius of a matrix is less than one if its power taken to infinite results an all zeros matrix (i.e. $\lim\limits_{k \rightarrow \infty} \left({\bf{F}}(t){\bf{F}}({t-1})\cdots {\bf{F}}({t-T})\right)^k = 0. {\bf{I}}$) \cite[pg. 618]{meyer04}. Thus, for any consecutive $T$ iterations, the spectral radius of the corresponding matrix ${\bf{F}}(t){\bf{F}}({t-1})\cdots{\bf{F}}({t-T})$ is strictly less than one. This also implies that as $k \gg T$, 
\begin{equation}
\begin{split}
\lim\limits_{k \rightarrow \infty} {\bf{F}}(k){\bf{F}}({k-1})\cdots {\bf{F}}(1) &= 
\lim\limits_{k \rightarrow \infty} \left({\bf{F}}(k)\cdots {\bf{F}}({k-T})\right)\left({\bf{F}}({k-T-1})    
\cdots {\bf{F}}({k-2T-1})\right)\cdots\\
&= 0.{\bf{I}}
\end{split}
\end{equation}
\end{proof}
In above, if each node updates at least once in every $T$ consecutive iterations then the product of matrices has a spectral radius less than one. Next, we show convergence of the matrix series in \eqref{asynclocaladapt} based on Theorem~\ref{jointSpecRad}.
\begin{theorem}
If $\rho\left({\bf{F}}\right)<1$, and each node updates at least once in every $T$ consecutive iterations, then the system converges. \label{convProof}
\end{theorem}
\begin{proof}
We can expand the series in (\ref{asynclocaladapt}) as follows:
\begin{eqnarray}
\begin{split} \label{asynseries}
{\bf{P}}(k+1) &={\bf{D}}(k)+{\bf{F}}({k-1}){\bf{D}}({k-1})+{\bf{F}}({k}){\bf{F}}({k-1}){\bf{D}}({k-2})\\
&+{\bf{F}}({k}){\bf{F}}({k-1}){\bf{F}}({k-2}){\bf{D}}({k-3})+\cdots\\
\end{split}
\end{eqnarray}
To show convergence of the series in (\ref{asynseries}), we take the \emph{absolute convergence} test \cite[pg. 181]{pugh2002}. That is, an infinite series of real numbers $\sum_{t}^{\infty} f_t$ will converge if the absolute of all its terms $\sum_t^{\infty} |f_t|$ converges. We first show the convergence of
\begin{equation}
\begin{split}
{\bf{D}}(k) +|{\bf{F}}({k-1})|{\bf{D}}({k-1}) +|{\bf{F}}({k}){\bf{F}}({k-1})|{\bf{D}}({k-2})+\\
|{\bf{F}}({k}){\bf{F}}({k-1}){\bf{F}}({k-2})|{\bf{D}}({k-3})+\cdots
\end{split}\label{asynproof}
\end{equation}
where, the terms ${\bf{D}}({k-i})$ can be separated out. 
Let us now consider a couple of intermediate steps in our proof. Firstly, a series of non-negative real numbers $\sum_{t}^{\infty}|f_t|$ will converge if a bounding series $\sum_t^{\infty} |q_t|$ such that $|q_t|\geq |f_t| \geq 0$ converges \cite[pg. 180]{pugh2002}. Secondly, a matrix series of the type ${\bf{I}}+{\bf{Q}}+{\bf{Q}}^2 + \cdots+ {\bf{Q}}^k$ converges if $\lim\limits_{k \rightarrow \infty} {\bf{Q}}^k = 0.{\bf{I}}$ \cite[pg. 126]{meyer04}. For some finite $T: T <k$
$$
\lim\limits_{k \rightarrow\infty}
|{\bf{F}}({k}){\bf{F}}({k-1})\cdots {\bf{F}}({1})| = 0.{\bf{I}}.
$$
Thus, for some arbitrary non-negative valued ${\bf{Q}}$ and ${\bf{R}}$ of dimensions $n \times n$ and $n \times 1$ respectively we will have 
\begin{equation}
\begin{split}
\lim\limits_{k \rightarrow\infty}{\bf{R}}+{\bf{Q}}{\bf{R}}+{\bf{Q}}^2{\bf{R}} + \cdots {\bf{Q}}^k{\bf{R}} &\geq \lim\limits_{k \rightarrow\infty} {\bf{D}}({k}) +|{\bf{F}}({k-1})|{\bf{D}}({k-1}) + \cdots \\
&|{\bf{F}}({k}){\bf{F}}({k-1})\cdots {\bf{F}}({1}) |{\bf{D}}({0}). 
\end{split}
\end{equation}
Thus this implies that the series in (\ref{asynproof}) will converge and so will the series in (\ref{asynseries}) as per the absolute convergence test.
\end{proof}



\subsection*{Estimation Error}
In practical systems, there may be an estimation error that renders node's knowledge about the system state as imperfect \cite{Medard2000}. The updates will thus be based on inaccurate state information.
Essentially, the estimation error will project ${{E}}(k)$ into $\widehat{E_i}(k)= \sum_{j\neq i} {\widehat{g_{i,i}}} P_j(k)$. Consequently, the matrix of the whole system can be denoted as $\widehat{\bf{F}}$ due to the imperfect or faulty state information remains constant over the time interval if the estimates are independent of the parameter updates. In that case, if we have $\rho(\widehat{\bf{F}})<1$ the system still convergences as Theorems~\ref{jointSpecRad} and ~\ref{convProof} still hold. 

\section{Rate of Convergence}
We let the probability mass function of the random variables $a_i(k)$ be based on an ergodic process and as follows:
\begin{equation}
a_i(k)=
\begin{cases}
0 \mbox{, $p_{ik}$,}\\
1 \mbox{, $1-p_{ik}$}
\end{cases}
\end{equation}
where $p_{ik}$ denotes the probability that node $i$ will update in iteration $k$. We define
\begin{equation}
\gamma = \underset{i,k}{\min} p_{ik}
\end{equation}
that denotes the lowest probability of an update by any node over the ergodic process. Therefore, the lower bound on the probability, denoted as $\lambda$, that in a $T$ iteration interval, all nodes will update at least once is
\begin{equation}
\lambda = \mbox{Prob($\sum_{t = k-T}^k{\bf{A}}(t) \geq {\bf{I}}$)}\geq (1-(1-\gamma)^T)^n
\end{equation}
given that the nodes are independent. The \emph{rate of convergence} $R$ may be described as the effective spectral radius over the $T$ iteration interval (i.e. a measure of $\frac{{\bf{P}}(k)-{\bf{P}}(k-T)}{T}$). Formally, it is defined as follows:
\begin{equation}
\begin{split}
R &:= \rho({\bf{F}})\frac{\lambda}{T}\\
&= \rho({\bf{F}})\frac{(1-(1-\gamma)^T)^n}{T}.
\end{split}
\end{equation}

Alternatively, if it is certain that each node will update in every iterations $T$, but the randomness is limited to the exact number of updates that each node will perform (i.e. ${\bf{I}} \leq \sum_{t = k-T}^k{\bf{A}}(t) \leq T.{\bf{I}}$), then the convergence rate is simply 
\begin{equation}
R := \frac{\rho({\bf{F}})}{T}
\end{equation}

\bibliography{dissertationBiblio}

%








\end{document}